\documentclass[reqno, 12pt]{article}

\pdfoutput=1

\usepackage{enumerate}
\usepackage{latexsym}
\usepackage[centertags]{amsmath}
\usepackage{amsfonts}
\usepackage{amssymb}
\usepackage{amsthm}
\usepackage{mathtools}
\usepackage{newlfont}
\usepackage{graphics}
\usepackage{color}
\usepackage{float}
\usepackage{diagbox}
\usepackage{tocloft}
\usepackage{titlesec}
\usepackage{booktabs}
\usepackage{extpfeil}
\usepackage{centernot}
\usepackage[pagebackref]{hyperref}
\usepackage[linesnumbered,ruled,vlined]{algorithm2e}
\usepackage{url}
\usepackage{hyperref}
\usepackage{longtable}
\usepackage{rotating}
\usepackage{multirow}
\usepackage{extarrows}
\usepackage[sort,compress,numbers]{natbib}
\usepackage[utf8]{inputenc}

\newtheorem{thm}{Theorem}[section]
\newtheorem{cor}[thm]{Corollary}
\newtheorem{lem}[thm]{Lemma}
\newtheorem{prop}[thm]{Proposition}

\theoremstyle{mydefinition}
\newtheorem{dfn}[thm]{Definition}
\theoremstyle{myremark}
\newtheorem{rem}[thm]{Remark}

\newtheorem{prob}[thm]{Open Problem}

\allowdisplaybreaks[4]

\SetKwInput{KwInput}{Input}                
\SetKwInput{KwOutput}{Output}              

\def\Z{\mathbb{Z}}

\title{Magic Positivity for the Ehrhart Polynomials of Partial Permutohedra}

\author{Feihu Liu$^{\color{blue} \dag}$ and Zihao Zhang$^{\color{blue} \S}$
\\[2mm]
{\small $^{\color{blue} \dag}$ Center for Combinatorics, LPMC}\\[-0.8ex]
{\small Nankai University, Tianjin 300071, P.R.~China}\\
{\small $^{\color{blue} \S}$ School of Mathematics and Statistics}\\[-0.8ex]
{\small Beijing Institute of Technology, Beijing 102400, P.R.~China}\\
{\small {\color{blue} $^\dag$} Email address: liufeihu7476@163.com}\\
{\small {\color{blue} $^\S$} Email address: zihao-zhang@foxmail.com}\\
}

\date{\today}

\begin{document}

\maketitle

\begin{abstract}
For positive integers \(m,n\), the partial permutohedron $\mathcal{P}(m,n)$ is a lattice polytope constructed as the convex hull of vectors in $\{0, 1, \dots, n\}^m$ that have distinct non-zero entries. We prove that for $n \ge m-1$, the Ehrhart polynomial of $\mathcal{P}(m,n)$ is magic positive except for the single case \((m,n)=(2,1)\). In particular, the Ehrhart polynomial of the parking function polytope (integrally equivalent to $\mathcal{P}(m,m-1)$) is magic positive for $m \ge 3$. For $n<m-1$, we discuss the magic positivity of the Ehrhart polynomial of $\mathcal{P}(m,n)$ for $n=1,2,3$. There exist infinitely many counterexamples with $n<m-1$ showing that the Ehrhart polynomial of $\mathcal{P}(m,n)$ is not magic positive. This partially resolves an open problem proposed by Ferroni and Higashitani.
\end{abstract}

\noindent
\begin{small}
\end{small}

\noindent
\begin{small}
\end{small}


\section{Introduction}

Let $\mathcal{P}$ be a \emph{lattice $d$-polytope}, i.e., a convex polytope in $\mathbb{R}^N$ whose vertices are elements of $\mathbb{Z}^N$ and whose affine span has dimension $d\leq N$.
The \emph{lattice point enumerator}
\[
i(\mathcal{P},t) = |t\mathcal{P} \cap \mathbb{Z}^N|, \quad t=1,2,\ldots,
\]
counts the number of integer points in the $t$-th dilation $t\mathcal{P} = \{ t\alpha : \alpha \in \mathcal{P} \}$ of $\mathcal{P}$.

Ehrhart~\cite{Ehrhart62} proved that the function $i(\mathcal{P},t)$ is a polynomial in $t$ of degree $d$ with constant term $1$.
Now $i(\mathcal{P},t)$ is called the \emph{Ehrhart polynomial} of $\mathcal{P}$.
One can find a polynomial $h_{\mathcal{P}}^{*}(x)$ of degree at most $d$ and satisfying
$$1+\sum_{t\geq 1}i(\mathcal{P},t)x^t=\frac{h_{\mathcal{P}}^{*}(x)}{(1-x)^{d+1}}.$$
The polynomial $h_{\mathcal{P}}^*(x)$ is often called as the \emph{$h^*$-polynomial} of $\mathcal{P}$.
Stanley~\cite{Stanleyh-polynomial} shows that the coefficients of $h^*$-polynomial are nonnegative integers.
For further background and knowledge on polytopes, we refer to several excellent books~\cite{BeckRobins}, \cite[Chapter 4]{RP.Stanley}, and \cite{Ziegler}.

A lattice polytope $\mathcal{P}$ is said to be \emph{Ehrhart positive} if all coefficients of $i(\mathcal{P},t)$ are non-negative.
For a comprehensive introduction to Ehrhart positivity, we refer the classic survey by Liu~\cite{FuLiu19}.
A stronger property than Ehrhart positivity is \emph{magic positivity}, defined as follows.
\begin{dfn}[Magic Positivity]
A lattice polytope $\mathcal{P}$ of dimension $d$ with Ehrhart polynomial $i(\mathcal{P},t)$ is said to be \emph{magic positive} if the polynomial can be expressed in the basis $\{t^i(t+1)^{d-i}\}_{i=0}^d$ with non-negative coefficients. That is,
\begin{align*}
i(\mathcal{P},t) = \sum_{i=0}^d a_i t^i(t+1)^{d-i},
\end{align*}
where $a_i \ge 0$ for all $0 \le i \le d$.
\end{dfn}

Clearly, the magic positivity of a lattice polytope implies the Ehrhart positivity of the polytope.
Another reason for studying the magic positivity of a lattice polytope $\mathcal{P}$ is that, by the result of Br\"and\'en \cite{Branden06}, it implies that the $h^*$-polynomial of $\mathcal{P}$ is real-rooted.

The main object of study in this paper is partial permutohedron. 
For positive integers \(m,n\), the \emph{partial permutohedron} \(\mathcal{P}(m,n)\subset \mathbb{R}^m\) is the convex hull of the vectors in \(\{0,1,\ldots,n\}^m\) whose nonzero entries are distinct. 
It is an \(m\)-dimensional lattice polytope.
The partial permutohedron was first introduced and studied by Heuer and Striker \cite{Heuer21,Heuer22}, and later further studied by Behrend et al \cite{Behedrn2023,Behedrn2025}.
Further such lattice polytopes were also studied by Black and Sanyal \cite{Black24} in the context of monotone path polytopes of polymatroids and by Hanada, Lentfer, and Vindas-Mel\'endez \cite{Hanada23} in the context of generalized parking function polytopes.
Recently, Behrend \cite{Behrend2024} gave a characterization for the Ehrhart polynomial of $\mathcal{P}(m,n)$ with $n \geq m-1$.

Ferroni and Higashitani mentioned the following open problem in their classic survey \cite{Ferroni23}.
\begin{prob}{\em \cite[Problem 4.23]{Ferroni23}}\label{Open-problem}
Characterize or classify all partial permutohedra that are magic positive.
\end{prob}

We note that partial permutohedra $\mathcal{P}(m,n)$ for $n\geq m-1$ are examples of $Y$-generalized permutohedra (under the unimodularly equivalent). 
The $Y$-generalized permutohedra is defined by Postnikov in \cite{PostnikovIMRN}.
This is a family of Ehrhart positive lattice polytopes.
Therefore, partial permutohedra $\mathcal{P}(m,n)$ for $n\geq m-1$ are Ehrhart positive.
Avila, Ferroni, and Morales \cite{Avila26} addressed the magic positivity problem for $Y$-generalized permutohedra under the assumption that the input parameters are sufficiently large.

The main contribution of this paper is to prove the following theorem.
\begin{thm}\label{thm:main}
Let \(m,n\in \mathbb{Z}_{>0}\) satisfy \(n\ge m-1\). Then the Ehrhart polynomial of the partial permutohedron $\mathcal{P}(m,n)$ is magic positive, except when \((m,n)=(2,1)\). In the exceptional case \(\mathcal{P}(2,1)\), the Ehrhart polynomial is not magic positive.
\end{thm}

A direct computation yields $h_{\mathcal{P}(2,1)}^{*}(x)=1$. Therefore, we obtain the following corollary.
\begin{cor}
Let \(m,n\in \mathbb{Z}_{>0}\) satisfy \(n\ge m-1\). Then the $h^*$-polynomial of the partial permutohedron $\mathcal{P}(m,n)$ is real-rooted.
\end{cor}

The partial permutohedron $\mathcal{P}(m,m-1)$, which is integrally equivalent to the well-known parking function polytope $P_m$.
A \emph{parking function} of length $m$ is an $m$-vector of positive integers whose nondecreasing rearrangement $a_1\leq a_2\leq \cdots \leq a_m$ satisfies $a_i\leq i$ for $i=1,\ldots,m$.
The \emph{parking function polytope} $P_m$ is the convex hull of all parking functions of length $m$.
For related work on the $P_m$, see \cite{Amanbayeva,Hanada23,Stanley2020AMM,Stong2022}.

\begin{cor}
For integer $m \ge 3$, the Ehrhart polynomial of the parking function polytope $P_m$ is magic positive.
\end{cor}

For the cases $n<m-1$, we prove the following small-\(n\) classification.

\begin{thm}\label{thm:main-Two}
In the range \(n<m-1\), the following statements hold.
\begin{enumerate}
\item For \(n=1\), the Ehrhart polynomial $i(\mathcal{P}(m,1),t)$ is not magic positive for every \(m\ge3\).

\item For \(n=2\), the Ehrhart polynomial $i(\mathcal{P}(m,2),t)$ is not magic positive for every \(m\ge4\).

\item For \(n=3\), the Ehrhart polynomial $i(\mathcal{P}(m,3),t)$ is magic positive exactly for \(m=5\) and \(m=6\).  It is not magic positive for every \(m\ge7\).
\end{enumerate}
\end{thm}

Thus the complementary range \(n<m-1\) is not uniformly magic positive; it contains both infinite families of counterexamples and, already for \(n=3\), small exceptional positive cases.
Theorems \ref{thm:main} and \ref{thm:main-Two} partially solves Open Problem \ref{Open-problem}.

The paper is organized as follows.
In Section \ref{Section-22}, we first convert the magic positivity condition into a standard polynomial non-negativity problem.
Then by the formula of Behrend \cite[Equation (30)]{Behrend2024}, we obtain a differential reduction in the parameter \(n\).
In Section \ref{Section-33}, we use rooted-tree function and Lagrange inversion to provide several lemmas required for the proof of the main theorem.
Section \ref{Section-44} is devoted to the proof of Theorem \ref{thm:main}.
In Section \ref{Section-55}, we discuss the magic positivity of $i(\mathcal{P}(m,1),t)$, $i(\mathcal{P}(m,2),t)$, and $i(\mathcal{P}(m,3),t)$ when $n<m-1$.

\section{The magic coefficients}\label{Section-22}

To prove Theorem \ref{thm:main}, we convert the magic positivity condition into a standard polynomial non-negativity problem via a rational substitution.

For a polynomial \(f(t)\) of degree at most \(d\), define its \emph{degree \(d\) magic transform} by
\[\mathsf{M}_d f(y)=(1-y)^d f\!\left(\frac{y}{1-y}\right).\]
Then $f(t)=\sum_{i=0}^d c_i\,t^i(1+t)^{d-i}$ if and only if $\mathsf{M}_d f(y)=\sum_{i=0}^d c_i y^i$.
Thus magic positivity is equivalent to coefficientwise nonnegativity of \( \mathsf{M}_d f(y)\).

We write $E_{m,n}(t)=i(\mathcal{P}(m,n),t)$ for the Ehrhart polynomial of partial permutohedron $\mathcal{P}(m,n)$.
The following coefficient-extraction formula of Behrend \cite[Equation (30)]{Behrend2024} will be the
starting point of the proof. 
For all positive integers \(m,n\) with \(n\ge m-1\),
\begin{align}\label{eq:Behrend-formula}
E_{m,n}(t)=m!\,t^m [z^m]\,\sqrt{1-z}\,\exp\!\left(\left(n+\frac12+\frac1t\right)z-\frac{z^2}{4t}\right).
\end{align}
Here \([z^m]\) denotes coefficient extraction.

We begin by converting Behrend's formula into an explicit expression for the magic coefficients.
Set
\[M_{m,n}(y)=\mathsf{M}_m E_{m,n}(y)=(1-y)^m E_{m,n}\!\left(\frac{y}{1-y}\right).\]
We obtain
\begin{align*}
M_{m,n}(y)&=(1-y)^mm!\left(\frac{y}{1-y}\right)^m[z^m]\,\sqrt{1-z}\,\exp\!\left(\left(n+\frac12+\frac{1-y}{y}\right)z
 -\frac{1-y}{y}\frac{z^2}{4}\right) 
\\&=m!\,y^m[z^m]\,\sqrt{1-z}\,\exp\!\left(\left(n-\frac12\right)z+\frac{z^2}{4}+\frac{z-z^2/4}{y}\right)                                                       \\&=m!\,[z^m]\,\sqrt{1-z}\,\exp\!\left(\left(n-\frac12\right)z+\frac{z^2}{4}\right)y^m\exp\!\left(\frac{z-z^2/4}{y}\right).
\end{align*}
Since \(z-z^2/4=z(1-z/4)\), the term
\((z-z^2/4)^r\) has \(z\)-degree at least \(r\).  Therefore, after
applying \([z^m]\), only the terms \(0\le r\le m\) can contribute:
\[y^m\exp\!\left(\frac{z-z^2/4}{y}\right)=\sum_{r=0}^m\frac{(z-z^2/4)^r}{r!}y^{m-r} \quad\text{modulo terms annihilated by }[z^m].\]
Hence
\[M_{m,n}(y)=\sum_{r=0}^m\frac{m!}{r!}[z^m]\,\sqrt{1-z}\,\exp\!\left(\left(n-\frac12\right)z+\frac{z^2}{4}\right)(z-z^2/4)^ry^{m-r}.\]
Writing $i=m-r$, $r=m-i$, and using $(z-z^2/4)^r=z^r(1-z/4)^r$, we get the following.

\begin{lem}\label{lem:magic-coefficients}
For \(0\le i\le m\), set \(r=m-i\).  The coefficient of \(y^i\) in \(M_{m,n}(y)\), equivalently the coefficient of \(t^i(1+t)^{m-i}\) in \(E_{m,n}(t)\), is
\begin{align}\label{eq:magic-coeff-B}
\mu_{m,n,i}=\frac{m!}{r!} B_{i,r}(n),
\end{align}
where
\begin{align}\label{eq:B-definition}
B_{i,r}(n)=[z^i]\,(1-z/4)^r\sqrt{1-z}\,\exp\!\left(\left(n-\frac12\right)z+\frac{z^2}{4}\right).
\end{align}
\end{lem}

Consequently, since \(m!/r!>0\), the magic coefficient \(\mu_{m,n,i}\) is nonnegative if and only if $B_{i,r}(n)\ge 0$.
The entire proof is therefore reduced to proving the nonnegativity of the quantities \(B_{i,r}(n)\) in the range $n\ge m-1=i+r-1$.

The first observation is that differentiation with respect to \(n\) lowers the first index of \(B_{i,r}\).

\begin{lem}\label{lem:derivative}
For \(i\ge 1\) and \(r\ge 0\), we have
\[\frac{\partial}{\partial n}B_{i,r}(n)=B_{i-1,r}(n).\]
Moreover,
\[B_{0,r}(n)=1\qquad\text{for all }r\ge 0.\]
\end{lem}
\begin{proof}
The parameter \(n\) occurs only in the exponential factor $\exp\!\left(\left(n-\frac12\right)z\right)$.
Hence, for \(i\ge 1\),
\begin{align*}
\frac{\partial}{\partial n}B_{i,r}(n)&=[z^i]\,(1-z/4)^r\sqrt{1-z}\,z\exp\!\left(\left(n-\frac12\right)z+\frac{z^2}{4}\right)                                                       \\&=[z^{i-1}]\,(1-z/4)^r\sqrt{1-z}\,\exp\!\left(\left(n-\frac12\right)z+\frac{z^2}{4}\right)                                                       \\&=B_{i-1,r}(n).
\end{align*}
Finally, \(B_{0,r}(n)\) is the constant term of a product whose constant term is \(1\), so \(B_{0,r}(n)=1\).
\end{proof}

Thus, once the boundary values $B_{i,r}(i+r-1)$ are known to be nonnegative, Lemma~\ref{lem:derivative} will propagate
nonnegativity to all \(n\ge i+r-1\).  The rest of the argument is devoted to proving precisely these boundary inequalities.

\section{Boundary values}\label{Section-33}

In this section, we characterize the boundary value $B_{i,r}(i+r-1)$ by means of rooted tree functions.
Then we discuss the nonnegativity of the boundary values.

\subsection{The rooted-tree function}
Let
\[T(u)=\sum_{k\ge 1} k^{k-1}\frac{u^k}{k!}\]
be the rooted-tree function.  It is characterized by the functional equation $T(u)=u e^{T(u)}$.
For further information on the rooted-tree function, we refer to \cite[Section 5]{RP.Stanley99}.
We shall use the following standard coefficient form of Lagrange inversion, see \cite{Corless96} and \cite{Gessel16}.
For completeness, we give a proof.

\begin{lem}\label{lem:Lagrange}
For every formal power series \(\phi(z)\) and every integer \(k\ge 0\),
\begin{align}\label{eq:Lagrange}
[u^k]\phi(T(u))=[z^k]\phi(z)(1-z)e^{kz}.
\end{align}
\end{lem}
\begin{proof}
The case \(k=0\) is immediate. For \(k\ge 1\), Lagrange inversion \cite[Theorem 2.1.1]{Gessel16} for the equation \(T=ue^T\) gives
\[ [u^k]\phi(T(u)) =\frac1k [z^{k-1}]\phi'(z)e^{kz}.\]
Now
\[\phi'(z)e^{kz} =\frac{d}{dz}\left(\phi(z)e^{kz}\right)-k\phi(z)e^{kz}.\]
Therefore
\begin{align*}
\frac1k [z^{k-1}]\phi'(z)e^{kz}&=\frac{1}{k} [z^{k-1}]\frac{d}{dz}\left(\phi(z)e^{kz}\right)-[z^{k-1}]\phi(z)e^{kz}                                      \\&=[z^k]\phi(z)e^{kz}-[z^{k-1}]\phi(z)e^{kz} 
\\&=[z^k]\phi(z)e^{kz}-[z^k]z\phi(z)e^{kz}
\\&=[z^k](1-z)\phi(z)e^{kz}.
\end{align*}
This proves \eqref{eq:Lagrange}.
\end{proof}

We now express the boundary values \(B_{i,r}(i+r-1)\) in terms of \(T(u)\).
Define
\begin{align}\label{eq:R-C-definition}
R(u)=\frac{T(u)(1-T(u)/4)}{u}\quad \text{and}\quad 
C(u)=\frac{\exp\!\left(-\frac32T(u)+\frac14T(u)^2\right)}{\sqrt{1-T(u)}}.
\end{align}

\begin{lem}\label{lem:boundary-RC}
For all \(i,r\ge 0\), we have the equality
\begin{align}\label{eq:boundary-RC}
B_{i,r}(i+r-1) =[u^i]\,R(u)^r C(u).
\end{align}
\end{lem}
\begin{proof}
By definition,
\[B_{i,r}(i+r-1)=[z^i]\,(1-z/4)^r\sqrt{1-z}\,\exp\!\left(\left(i+r-\frac32\right)z+\frac{z^2}{4}\right).\]
Apply Lemma~\ref{lem:Lagrange} with \(k=i\) and
\[\phi(z)=(1-z/4)^r(1-z)^{-1/2}\exp\!\left(\left(r-\frac32\right)z+\frac{z^2}{4}\right).\]
Then
\begin{align*}
[u^i]\phi(T(u))&=[z^i]\,(1-z)\phi(z)e^{iz}
\\&=[z^i]\,(1-z/4)^r\sqrt{1-z}\,\exp\!\left(\left(i+r-\frac32\right)z+\frac{z^2}{4}\right)                                                      \\&=B_{i,r}(i+r-1).
\end{align*}
It remains to rewrite \(\phi(T(u))\).  Since \(T(u)=ue^{T(u)}\), we have $e^{rT(u)} =\left(\frac{T(u)}{u}\right)^r$.
Thus
\begin{align*}
\phi(T(u))&=(1-T(u)/4)^r(1-T(u))^{-1/2}\exp\!\left(\left(r-\frac32\right)T(u)+\frac14T(u)^2\right)                                                      \\&=\left(\frac{T(u)(1-T(u)/4)}{u}\right)^r\frac{\exp\!\left(-\frac32T(u)+\frac14T(u)^2\right)}{\sqrt{1-T(u)}}                                                           \\&=R(u)^r C(u).
\end{align*}
Therefore
\[B_{i,r}(i+r-1)=[u^i]R(u)^rC(u),\]
as claimed.
\end{proof}

The proof of boundary nonnegativity is now reduced to coefficient
properties of \(R(u)\) and \(C(u)\).

\subsection{Coefficient estimates for \(R(u)\)}

We first determine the coefficients of \(R(u)\).

\begin{lem}\label{lem:R-coefficients}
Write $R(u)=\sum_{k\ge 0}\rho_k u^k$. Then $\rho_0=1$, and, for \(k\ge 1\),
\begin{equation}\label{eq:rho-formula}
        \rho_k
        =
        \frac{(k+2)(k+1)^{k-2}}{2k!}.
\end{equation}
In particular, $\rho_k\ge 0$  for all $k\ge 0$, and $\rho_k\ge \rho_{k-1}$ for all $k\ge 2$.
\end{lem}
\begin{proof}
Since \(T(u)=ue^{T(u)}\), we have $R(u)=e^{T(u)}(1-T(u)/4)$.
For \(k\ge 0\), Lemma~\ref{lem:Lagrange} gives
\[\rho_k=[u^k]e^{T(u)}(1-T(u)/4)=[z^k](1-z)(1-z/4)e^{(k+1)z}.\]
For \(k=0\), this gives \(\rho_0=1\).  For \(k=1\),
\[\rho_1= [z](1-z)(1-z/4)e^{2z} =2-\frac{5}{4} =\frac{3}{4}.\]
For \(k\ge 2\),
\begin{align*}
\rho_k&=[z^k]\left(1-\frac54z+\frac14z^2\right)e^{(k+1)z}
=\frac{(k+1)^k}{k!}-\frac54\frac{(k+1)^{k-1}}{(k-1)!}+\frac14\frac{(k+1)^{k-2}}{(k-2)!} 
\\ &=\frac{(k+1)^{k-2}}{k!}\left((k+1)^2-\frac54 k(k+1)+\frac14 k(k-1)\right)                                                      \\ &=\frac{(k+2)(k+1)^{k-2}}{2k!}.
\end{align*}
The same closed formula also gives \(\rho_1=3/4\), so \eqref{eq:rho-formula} holds for all \(k\ge 1\).

It is immediate from \eqref{eq:rho-formula} that \(\rho_k\ge 0\).  For
\(k\ge 2\),
\[\frac{\rho_k}{\rho_{k-1}}=\frac{k+2}{k}\left(1+\frac1k\right)^{k-3}.
\]
For \(k=2\), this ratio is \(4/3\).  For \(k\ge 3\), both factors on the
right-hand side are at least \(1\).  Hence $\rho_k\ge \rho_{k-1}$ for all $k\ge 2$.
\end{proof}

Although \(R(u)\) itself is not coefficientwise nondecreasing from the
constant term, because \(\rho_1=3/4<1=\rho_0\), its powers \(R(u)^r\)
are coefficientwise nondecreasing for \(r\ge 2\).  We prove this next.

\begin{lem}\label{lem:R-powers-monotone}
Let $R(u)^r=\sum_{k\ge 0}q_k^{(r)}u^k$. If \(r\ge 2\), then $q_k^{(r)}\ge q_{k-1}^{(r)}$ for all $k\ge 1$.
\end{lem}
\begin{proof}
We first prove the claim for \(r=2\).  Write $R(u)^2=\sum_{k\ge 0}q_k u^k$. 
Since $R(u)^2=e^{2T(u)}(1-T(u)/4)^2$, Lemma~\ref{lem:Lagrange} gives
\[q_k=[z^k](1-z)(1-z/4)^2e^{(k+2)z}.\]
Clearly \(q_0=1\).  Also $q_1=\frac{3}{2}$, $q_2=\frac{41}{16}$.
For \(k\ge 3\), expanding
\[(1-z)(1-z/4)^2=1-\frac32z+\frac9{16}z^2-\frac1{16}z^3\]
gives
\begin{align*}
q_k&=\frac{(k+2)^k}{k!}-\frac32\frac{(k+2)^{k-1}}{(k-1)!}+\frac9{16}\frac{(k+2)^{k-2}}{(k-2)!}-\frac1{16}\frac{(k+2)^{k-3}}{(k-3)!}                         \\&=\frac{(k+2)^{k-3}}{k!}\left((k+2)^3-\frac32 k(k+2)^2+\frac9{16}k(k-1)(k+2)-\frac1{16}k(k-1)(k-2)\right)                                                      \\&=\frac{(k+2)^{k-3}}{k!}\cdot\frac{3k^2+19k+32}{4}.
\end{align*}
The same formula also agrees with the values \(q_1=3/2\) and
\(q_2=41/16\).  Hence, for all \(k\ge 1\),
\begin{align}\label{eq:q-formula}
q_k =\frac{(k+2)^{k-3}(3k^2+19k+32)}{4k!}.
\end{align}

We now check that \(q_k\ge q_{k-1}\).  The inequalities $q_1=\frac32\ge 1=q_0$, $q_2=\frac{41}{16}\ge \frac32=q_1$ are immediate.  For \(k\ge 3\), \eqref{eq:q-formula} gives
\[\frac{q_k}{q_{k-1}} =\frac{k+1}{k}\left(1+\frac1{k+1}\right)^{k-3}\frac{3k^2+19k+32}{3k^2+13k+16}.
\]
Each of the three factors on the right-hand side is at least \(1\), and the first and third are strictly greater than \(1\).  Thus
\[q_k\ge q_{k-1}\qquad\text{for all }k\ge 1.\]
So \(R(u)^2\) has nonnegative and coefficientwise nondecreasing coefficients.

It remains to pass from \(R^2\) to \(R^r\) for \(r\ge 2\).  We use the
following elementary fact.  If
\[A(u)=\sum_{k\ge 0}a_k u^k\]
has nonnegative and coefficientwise nondecreasing coefficients, and
\[B(u)=\sum_{k\ge 0}b_k u^k\]
has nonnegative coefficients, then \(A(u)B(u)\) has coefficientwise nondecreasing coefficients.  Indeed, if
\[A(u)B(u)=\sum_{k\ge 0}c_k u^k,\]
then, for \(k\ge 1\),
\begin{align*}
c_k-c_{k-1} =\sum_{\ell=0}^k b_\ell a_{k-\ell}-\sum_{\ell=0}^{k-1}b_\ell a_{k-1-\ell}
=b_k a_0+\sum_{\ell=0}^{k-1}b_\ell\bigl(a_{k-\ell}-a_{k-1-\ell}\bigr)\ge 0.
\end{align*}
By Lemma~\ref{lem:R-coefficients}, \(R(u)\) has nonnegative coefficients.  Therefore \(R(u)^{r-2}\) has nonnegative coefficients.
Since $R(u)^r=R(u)^2R(u)^{r-2}$, the preceding product argument implies that \(R(u)^r\) has coefficientwise nondecreasing coefficients for every \(r\ge 2\).
\end{proof}

\subsection{Coefficient estimates for \(C(u)\)}

We next prove the required coefficient statement for \(C(u)\).

\begin{lem}\label{lem:C-coefficients}
There exists a formal power series \(P(u)\) with nonnegative coefficients such that
\begin{align}\label{eq:C-1-u-P}
C(u)=1-u+P(u).
\end{align}
Equivalently, $[u]C(u)=-1$, and all coefficients of \(C(u)\) of degree at least \(2\) are nonnegative.
\end{lem}
\begin{proof}
Define
\[G(w)=(1-w)^{-1/2}\exp\!\left(-\frac w2+\frac{w^2}{4}\right).\]
Then
\begin{align*}
\log G(w)=-\frac12\log(1-w)-\frac w2+\frac{w^2}{4}
=\frac12\sum_{j\ge 1}\frac{w^j}{j}-\frac w2+\frac{w^2}{4}
=\frac{w^2}{2}+\sum_{j\ge 3}\frac{w^j}{2j}.
\end{align*}
Thus \(\log G(w)\) has no constant or linear term, and all its coefficients are nonnegative.  It follows that
\[G(w)=\sum_{q\ge 0}g_q w^q\]
has
\[g_0=1,\qquad g_1=0, \qquad g_q\ge 0\quad(q\ge 2).
\]
Moreover, from the displayed expression for \(\log G(w)\),
\[g_2=\frac12,\qquad g_3=\frac16, \qquad g_4=\frac14.
\]
By the definition of \(C(u)\), $C(u)=e^{-T(u)}G(T(u))$. 
Since \(T(u)=ue^{T(u)}\), we have $e^{-T(u)}=\frac{u}{T(u)}$.
Therefore
\[e^{-T(u)}T(u)^q=uT(u)^{q-1}\qquad(q\ge 1),
\]
and hence
\begin{align}\label{eq:C-expanded}
C(u)=e^{-T(u)}+\sum_{q\ge 2}g_q\,uT(u)^{q-1}.
\end{align}

We shall use two standard coefficient evaluations. First, by
Lemma~\ref{lem:Lagrange}, for \(k\ge 1\),
\begin{align}
[u^k]e^{-T(u)}=[z^k](1-z)e^{(k-1)z}=\frac{(k-1)^k}{k!}-\frac{(k-1)^{k-1}}{(k-1)!} =-\frac{(k-1)^{k-1}}{k!}. \label{eq:e-minus-T-coeff}
\end{align}
Second, for integers \(N\ge a\ge 1\), ordinary Lagrange inversion \cite[Equation (2.2.1)]{Gessel16} gives
\begin{align}\label{eq:T-power-coeff}
[u^N]T(u)^a=\frac{a}{N}[z^{N-a}]e^{Nz}=\frac{a}{N}\frac{N^{N-a}}{(N-a)!}.
\end{align}

Let $C(u)=\sum_{k\ge 0}c_k u^k$. It is clear that \(c_0=1\).  Using \eqref{eq:C-expanded} and
\eqref{eq:e-minus-T-coeff}, the coefficient of \(u\) is $c_1=[u]e^{-T(u)}=-1$. 
For \(k=2\), only the terms \(e^{-T(u)}\) and \(g_2uT(u)\) contribute:
\[c_2=-\frac12+\frac12[u]T(u)=-\frac12+\frac12=0.
\]
For \(k=3\), only the terms \(e^{-T(u)}\), \(g_2uT(u)\), and
\(g_3uT(u)^2\) contribute:
\[c_3=-\frac{2^2}{3!}+\frac12[u^2]T(u)+\frac16[u^2]T(u)^2=-\frac23+\frac12+\frac16=0.
\]

Now let \(k\ge 4\).  Since \(g_q\ge 0\) for all \(q\ge 2\), and since
\(T(u)\) has nonnegative coefficients, all terms in \eqref{eq:C-expanded} with \(q\ge 5\) contribute nonnegatively to
\([u^k]C(u)\).  Therefore, using only the terms \(q=2,3,4\), we get the lower bound
\begin{align*}
c_k&\ge-\frac{(k-1)^{k-1}}{k!}+\frac12 [u^{k-1}]T(u)+\frac16 [u^{k-1}]T(u)^2+\frac14 [u^{k-1}]T(u)^3                                    \\ &=-\frac{(k-1)^{k-1}}{k!}+\frac12\frac{(k-1)^{k-2}}{(k-1)!}+\frac13\frac{(k-1)^{k-4}}{(k-3)!} +\frac34\frac{(k-1)^{k-5}}{(k-4)!}
\\ &= \frac{(k-1)^{k-5}}{k!}\cdot \frac{(k-3)(k-2)(k-1)(7k+2)}{12}.
\end{align*}
Therefore $c_k\ge 0$ for $k\ge 4$.
Thus $C(u)=1-u+P(u)$ for a power series \(P(u)\) with nonnegative coefficients.
\end{proof}

\subsection{Boundary nonnegativity}

We now combine the coefficient estimates for \(R(u)\) and \(C(u)\).

\begin{prop}\label{prop:boundary-nonnegative}
Let \(i,r\ge 0\) satisfy \(i+r\ge 2\).  Then $B_{i,r}(i+r-1)\ge 0$ 
except for the single pair \((i,r)=(1,1)\).  For this exceptional pair, $B_{1,1}(1)=-\frac{1}{4}$.
\end{prop}
\begin{proof}
By Lemma~\ref{lem:boundary-RC}, $B_{i,r}(i+r-1)=[u^i]R(u)^rC(u)$. 
By Lemma~\ref{lem:C-coefficients}, $C(u)=1-u+P(u)$, 
where \(P(u)\) has nonnegative coefficients.  Write
\[R(u)^r=\sum_{k\ge 0}q_k^{(r)}u^k,
\]
with the convention \(q_{-1}^{(r)}=0\).  Since \(R(u)\) has nonnegative coefficients by Lemma~\ref{lem:R-coefficients}, the product
\(R(u)^rP(u)\) has nonnegative coefficients.  Therefore
\begin{align}\label{eq:boundary-main-ineq}
[u^i]R(u)^rC(u)=q_i^{(r)}-q_{i-1}^{(r)}+[u^i]R(u)^rP(u).
\end{align}
We consider three cases.
First, suppose \(r=0\).  Then \(R(u)^r=1\).  Since \(i+r\ge 2\), we
have \(i\ge 2\).  Hence, by Lemma~\ref{lem:C-coefficients},
\[[u^i]R(u)^rC(u)=[u^i]C(u)\ge 0.\]

Second, suppose \(r=1\).  Then \(q_k^{(1)}=\rho_k\), where $R(u)=\sum_{k\ge 0}\rho_k u^k$.
If \(i=1\), then by Lemma~\ref{lem:R-coefficients}, $q_1^{(1)}-q_0^{(1)}=\rho_1-\rho_0=\frac{3}{4}-1=-\frac{1}{4}$.
Moreover, \(P(u)\) has no constant or linear term, so $[u]R(u)P(u)=0$.
Thus $[u]R(u)C(u)=-\frac{1}{4}$.
This is exactly the exceptional pair \((i,r)=(1,1)\).

If instead \(r=1\) and \(i\ge 2\), then Lemma~\ref{lem:R-coefficients}
gives $\rho_i-\rho_{i-1}\ge 0$. 
The final term in \eqref{eq:boundary-main-ineq} is also nonnegative, so $[u^i]R(u)C(u)\ge 0$.

Third, suppose \(r\ge 2\).  By Lemma~\ref{lem:R-powers-monotone}, $q_i^{(r)}-q_{i-1}^{(r)}\ge 0$.
The term \([u^i]R(u)^rP(u)\) in \eqref{eq:boundary-main-ineq} is
nonnegative.  Hence $[u^i]R(u)^rC(u)\ge 0$. 

The proposition follows.
\end{proof}

\section{Proof of the main theorem}\label{Section-44}

We now extend the boundary inequalities from \(n=i+r-1\) to the entire range \(n\ge i+r-1\).

\begin{prop}\label{prop:B-nonnegative}
Let \(i,r\ge 0\) satisfy \(i+r\ge 2\), and suppose $(i,r)\ne (1,1)$. Then
\[B_{i,r}(n)\ge 0 \qquad \text{for every }n\ge i+r-1.\]
In addition, $B_{1,1}(n)=n-\frac54$, so in particular $B_{1,1}(n)\ge 0$ for every $n\ge 2$.
\end{prop}
\begin{proof}
First observe directly from \eqref{eq:B-definition} that
\[B_{1,0}(n)=[z]\sqrt{1-z}\,\exp\!\left(\left(n-\frac12\right)z+\frac{z^2}{4}\right)=n-1,
\]
and
\[B_{1,1}(n)=[z](1-z/4)\sqrt{1-z}\,\exp\!\left(\left(n-\frac12\right)z+\frac{z^2}{4}\right)=n-\frac54.
\]
Thus
\[B_{1,0}(n)\ge 0\quad(n\ge 1),\qquad B_{1,1}(n)\ge 0\quad(n\ge 2).
\]
We prove the main assertion by induction on \(i\), simultaneously for
all \(r\ge 0\).  If \(i=0\), then Lemma~\ref{lem:derivative} gives $B_{0,r}(n)=1$, so the assertion is immediate.

Assume \(i\ge 1\), and suppose the proposition has already been proved for all smaller first indices.  Set $N=i+r-1$.
By Proposition~\ref{prop:boundary-nonnegative}, $B_{i,r}(N)\ge 0$, because \(i+r\ge 2\) and \((i,r)\ne(1,1)\).

We claim that $B_{i-1,r}(x)\ge 0$ for all $x\ge N$. There are three possibilities.

If
\[(i-1)+r\ge 2\quad\text{and} \quad(i-1,r)\ne (1,1),
\]
then the induction hypothesis gives
\[B_{i-1,r}(x)\ge 0\qquad\text{for all }x\ge (i-1)+r-1=N-1.
\]
In particular, this holds for all \(x\ge N\).
If \((i-1,r)=(1,1)\), then \((i,r)=(2,1)\) and \(N=2\).  In this case
the direct computation above gives
\[B_{i-1,r}(x)=B_{1,1}(x)=x-\frac54\ge 0\qquad\text{for all }x\ge 2=N.
\]

The only remaining possibility is that $(i-1)+r<2$. Since \(i+r\ge 2\) and \(i\ge 1\), this can only occur for
\((i,r)=(2,0)\) or \((i,r)=(1,1)\).  The latter is excluded.  Hence \((i,r)=(2,0)\), so \(N=1\), and the direct computation gives
\[B_{i-1,r}(x)=B_{1,0}(x)=x-1\ge 0 \qquad\text{for all }x\ge 1=N.
\]
Thus in every case,
\[B_{i-1,r}(x)\ge 0 \qquad\text{for all }x\ge N.
\]

Using Lemma~\ref{lem:derivative}, for \(n\ge N\) we have
\[B_{i,r}(n)=B_{i,r}(N)+\int_N^n\frac{\partial}{\partial x}B_{i,r}(x)\,dx=B_{i,r}(N)+\int_N^n B_{i-1,r}(x)\,dx.
\]
Both terms on the right-hand side are nonnegative.  Therefore
\[B_{i,r}(n)\ge 0\qquad\text{for all }n\ge N=i+r-1.
\]
This completes the induction.
\end{proof}

\begin{proof}[Proof of Theorem~\ref{thm:main}]
First consider \(m=1\).  Then $\mathcal{P}(1,n)=[0,n]$, so $E_{1,n}(t)=nt+1$.
In the magic basis \(\{1+t,t\}\), $nt+1=(1+t)+(n-1)t$.
Since \(n\in\Z_{>0}\), both coefficients are nonnegative.  Thus \(E_{1,n}(t)\) is magic positive.

Now assume \(m\ge 2\).  By Lemma~\ref{lem:magic-coefficients}, the coefficient of \(t^i(1+t)^{m-i}\) in \(E_{m,n}(t)\) is
\[\mu_{m,n,i}=\frac{m!}{r!}B_{i,r}(n),\qquad r=m-i.
\]
The scalar \(m!/r!\) is positive.  Therefore it suffices to prove $B_{i,r}(n)\ge 0$ for every \(0\le i\le m\), where \(r=m-i\).

Since \(i+r=m\), the hypothesis \(n\ge m-1\) is exactly $n\ge i+r-1$.
If \((i,r)\ne(1,1)\), then Proposition~\ref{prop:B-nonnegative} gives $B_{i,r}(n)\ge 0$.
The only remaining index pair is \((i,r)=(1,1)\), which can occur only when \(m=2\).  For this pair, $B_{1,1}(n)=n-\frac{5}{4}$.
Thus
\[B_{1,1}(n)\ge 0\qquad\text{whenever } n\ge 2.
\]
Consequently, if \((m,n)\ne(2,1)\), every magic coefficient of \(E_{m,n}(t)\) is nonnegative.  This proves magic positivity for all
non-exceptional cases.

It remains to verify that \((m,n)=(2,1)\) is genuinely exceptional.
For \(m=2\), the middle magic coefficient corresponds to \(i=1\) and
\(r=1\).  Hence
\[\mu_{2,1,1}=\frac{2!}{1!}B_{1,1}(1)=2\left(1-\frac{5}{4}\right)=-\frac{1}{2}.
\]
Therefore \(E_{2,1}(t)\) is not magic positive.

This completes the proof.
\end{proof}

\begin{rem}
The proof uses the hypothesis \(n\ge m-1\) in two essential places.
First, this is the range in which the coefficient-extraction formula \eqref{eq:Behrend-formula} is available.  Second, after writing
\(r=m-i\), the lower endpoint of the allowed range is exactly $n=m-1=i+r-1$,
which is the boundary point at which the coefficient estimates above are
proved.
\end{rem}

\section{The Cases $n<m-1$}\label{Section-55}

This section studies the complementary range \(n<m-1\) for small \(n\).  We prove that \(i(\mathcal{P}(m,1),t)\) is not magic positive for all \(m\ge3\), that \(i(\mathcal{P}(m,2),t)\) is not magic positive for all \(m\ge4\), and that, in the range \(n=3<m-1\), the only magic positive cases are \(\mathcal{P}(5,3)\) and \(\mathcal{P}(6,3)\).

\begin{prop}{\em \cite[Section 5.3]{Behedrn2025}}\label{Proposition-123E}
For any positive integer $m$, the Ehrhart polynomials of $\mathcal{P}(m,1)$, $\mathcal{P}(m,2)$, and $\mathcal{P}(m,3)$ are respectively as follows.
\begin{align*}
E_{m,1}(t)&=\binom{t+m}{m},
\\ E_{m,2}(t)&=\binom{3t+m}{m}-m\binom{t+m-1}{m},
\\ E_{m,3}(t)&=\binom{6t+m}{m}-m\binom{3t+m-1}{m}-\binom{m}{2}\left(\binom{t+m-1}{m}+(m-2)\binom{t+m-2}{m}\right).
\end{align*}
\end{prop}

\begin{lem}\label{lem:first-magic-coeff}
Let $f(t)=1+e_1t+e_2t^2+\cdots+e_m t^m$. If $f(t)=\sum_{k=0}^{m} a_k t^k(1+t)^{m-k}$, then
$a_0=1$, $a_1=e_1-m$.
\end{lem}
\begin{proof}
The constant term of \(f(t)\) is \(1\). Therefore, $a_0=1$.
Observing the coefficient of the linear term of $f(t)$, we obtain $e_1=m a_0+a_1=m+a_1$.
\end{proof}

We shall also need a closed formula for the magic transform of a binomial polynomial.

\begin{lem}\label{lem:Phi}
Let
\[\Phi_m(\alpha,\beta;y)=(1-y)^m\binom{\alpha \frac{y}{1-y}+\beta}{m}.
\]
Then
\[\Phi_m(\alpha,\beta;y)=\frac1{m!}\prod_{j=0}^{m-1}\left(\beta-j+(\alpha-\beta+j)y\right).
\]
\end{lem}
\begin{proof}
By definition, $\binom{x}{m}=\frac1{m!}\prod_{j=0}^{m-1}(x-j)$.
Taking $x=\alpha\frac{y}{1-y}+\beta$ gives
\[\Phi_m(\alpha,\beta;y)=\frac{(1-y)^m}{m!}\prod_{j=0}^{m-1}\left(\alpha\frac{y}{1-y}+\beta-j\right)
=\frac1{m!}\prod_{j=0}^{m-1}\left(\beta-j+(\alpha-\beta+j)y\right).
\]
This completes the proof.
\end{proof}

For later convenience, define
\[P_{\gamma,q}(y)=\prod_{s=1}^{q}\left(1+\left(\frac{\gamma}{s}-1\right)y\right),\qquad q\ge0,
\]
where the empty product is \(1\).  Its first coefficient is
\[[y]P_{\gamma,q}(y)=\sum_{s=1}^{q}\left(\frac{\gamma}{s}-1\right)=\gamma H_q-q,
\]
where
\[H_q=1+\frac12+\cdots+\frac1q
\]
is the \(q\)-th \emph{harmonic number}, and \(H_0:=0\).

\subsection{The case \(n=1\)}

When \(n=1\), the polytope $\mathcal{P}(m,1)$ is the standard simplex.

\begin{prop}\label{prop:n1}
For \(m\ge2\), the coefficient of \(y\) in the degree \(m\) magic transform of $E_{m,1}(y)$ (i.e., $[y]\mathsf{M}_m E_{m,1}(y)$) is
\[H_m-m.\]
Consequently, \(E_{m,1}(t)\) is not magic positive for every \(m\ge2\).
\end{prop}
\begin{proof}
By Proposition \ref{Proposition-123E} and Lemma~\ref{lem:Phi} with \(\alpha=1\) and \(\beta=m\), we get
\[\mathsf{M}_m E_{m,1}(y) = \frac1{m!} \prod_{j=0}^{m-1} \left( m-j+(1-m+j)y\right).
\]
Letting \(s=m-j\), so that \(s=1,\ldots,m\), this becomes
\[\mathsf{M}_m E_{m,1}(y)=\frac1{m!}\prod_{s=1}^{m}\left(s+(1-s)y\right)=\prod_{s=1}^{m}\left(1+\left(\frac1s-1\right)y\right)
=P_{1,m}(y).
\]
Therefore $[y]\mathsf{M}_m E_{m,1}(y)=H_m-m$.
For \(m\ge2\),
\[H_m=1+\sum_{s=2}^{m}\frac1s<1+\sum_{s=2}^{m}1=m.
\]
Hence $H_m-m<0$. 
Thus the magic transform has a negative coefficient, and \(E_{m,1}(t)\) is not magic positive.
\end{proof}

\subsection{The case \(n=2\)}

We next treat \(n=2\).  In the range \(n<m-1\), this means \(m\ge4\).
We now compute its magic transform.

\begin{prop}\label{prop:M-n2}
For \(m\ge4\), we have $\mathsf{M}_m E_{m,2}(y)=P_{3,m}(y)-yP_{1,m-1}(y)$.
Consequently, we obtain 
$$[y]\mathsf{M}_m E_{m,2}(y) = 3H_m-m-1.$$
\end{prop}
\begin{proof}
By Proposition \ref{Proposition-123E} and Lemma~\ref{lem:Phi},
\[\mathsf{M}_m E_{m,2}(y) = \Phi_m(3,m;y) - m\Phi_m(1,m-1;y).
\]
The first term is
\[\Phi_m(3,m;y)=\frac1{m!}\prod_{j=0}^{m-1}\left( m-j+(3-m+j)y\right).
\]
Setting \(s=m-j\) gives
\[\Phi_m(3,m;y) = \prod_{s=1}^{m} \left( 1+\left(\frac3s-1\right)y \right) = P_{3,m}(y).
\]
For the second term, Lemma~\ref{lem:Phi} gives
\[\Phi_m(1,m-1;y)=\frac1{m!}\prod_{j=0}^{m-1}\left(m-1-j+(1-m+1+j)y\right).
\]
The factor with \(j=m-1\) is \(y\).  The remaining factors correspond to \(s=1,\ldots,m-1\) and give
\[\Phi_m(1,m-1;y)=\frac{(m-1)!}{m!}y\prod_{s=1}^{m-1}\left(1+\left(\frac1s-1\right)y\right).
\]
Thus $m\Phi_m(1,m-1;y) = yP_{1,m-1}(y)$.
Hence
\[ \mathsf{M}_m E_{m,2}(y) = P_{3,m}(y)-yP_{1,m-1}(y).
\]
Taking the coefficient of \(y\), we obtain $[y]P_{3,m}(y)=3H_m-m$, while $[y]\bigl(yP_{1,m-1}(y)\bigr)=1$.
Therefore
\[ [y]\mathsf{M}_m E_{m,2}(y) = 3H_m-m-1.
\]
This completes the proof.
\end{proof}

\begin{prop}\label{prop:n2-not-positive}
For every \(m\ge4\), the polynomial \(E_{m,2}(t)\) is not magic positive.
\end{prop}
\begin{proof}
For \(m\ge7\), Proposition~\ref{prop:M-n2} gives $[y]\mathsf{M}_m E_{m,2}(y)=3H_m-m-1$.
For \(m=7\),
\[3H_7-8=3\left(1+\frac12+\frac13+\frac14+\frac15+\frac16+\frac17\right)-8=-\frac{31}{140}<0.
\]
Moreover,
\[ \bigl(3H_{m+1}-(m+1)-1\bigr) - \bigl(3H_m-m-1\bigr) = \frac3{m+1}-1<0
\]
for every \(m\ge3\).  Hence $3H_m-m-1<0$ with $(m\ge7)$.
Thus \(E_{m,2}(t)\) is not magic positive for \(m\ge7\).
It remains to check \(m=4,5,6\).  Using the exact product formula
\[\mathsf{M}_m E_{m,2}(y) = P_{3,m}(y)-yP_{1,m-1}(y),
\]
one obtains
\begin{align*}
\mathsf{M}_4 E_{4,2}(y)&=1+\frac54 y+\frac{37}{24}y^2-\frac{7}{12}y^3,
\\ \mathsf{M}_5 E_{5,2}(y)&=1+\frac{17}{20}y+\frac{167}{120}y^2-\frac{193}{120}y^3+\frac7{20}y^4,
\\ \mathsf{M}_6 E_{6,2}(y)&= 1+\frac7{20}y+\frac{19}{15}y^2-\frac{691}{240}y^3 +\frac{91}{60}y^4-\frac14y^5.
\end{align*}
Each of these polynomials has a negative coefficient. Therefore \(E_{m,2}(t)\) is not magic positive for \(m=4,5,6\) as well.
\end{proof}

\subsection{The case \(n=3\)}

We now treat \(n=3\).  In the range \(n<m-1\), this means \(m\ge5\).
We now transform the formula into the magic basis.

\begin{prop}\label{prop:M-n3}
For \(m\ge5\),
\[\mathsf{M}_m E_{m,3}(y)=P_{6,m}(y)-3yP_{3,m-1}(y)-\frac{m-1}{2}yP_{1,m-1}(y)+\frac{m-2}{2}y(1-2y)P_{1,m-2}(y).
\]
Consequently,
\[[y]\mathsf{M}_m E_{m,3}(y)=6H_m-m-\frac72.
\]
\end{prop}
\begin{proof}
By Proposition \ref{Proposition-123E} and Lemma~\ref{lem:Phi},
\begin{align*}
\mathsf{M}_m E_{m,3}(y) &=\Phi_m(6,m;y)-m\Phi_m(3,m-1;y)-\binom m2\Phi_m(1,m-1;y)
\\ &\qquad\qquad -\binom m2(m-2)\Phi_m(1,m-2;y).
\end{align*}
The first term is $\Phi_m(6,m;y)=P_{6,m}(y)$.
For the second term, the factor with \(j=m-1\) contributes \(3y\), and
the remaining factors give \(P_{3,m-1}(y)\).  Since $m\cdot\frac{(m-1)!}{m!}=1$,
we get
\[ m\Phi_m(3,m-1;y)=3yP_{3,m-1}(y).\]
For the third term,
\[\Phi_m(1,m-1;y)= \frac{(m-1)!}{m!}yP_{1,m-1}(y) = \frac1m yP_{1,m-1}(y).
\]
Thus
\[\binom m2\Phi_m(1,m-1;y) = \frac{m-1}{2}yP_{1,m-1}(y).
\]
For the fourth term, apply Lemma~\ref{lem:Phi} with \(\alpha=1\) and \(\beta=m-2\).  
The factor with \(j=m-2\) and $j=m-1$ contributes respectively $y$ and $-1+2y$.
The remaining factors give $(m-2)!P_{1,m-2}(y)$.
Therefore
\[\Phi_m(1,m-2;y)=-\frac{(m-2)!}{m!}y(1-2y)P_{1,m-2}(y)=-\frac{1}{m(m-1)}y(1-2y)P_{1,m-2}(y).
\]
Multiplying by \(-\binom{m}{2}(m-2)\), we get
\[-\binom m2(m-2)\Phi_m(1,m-2;y)=\frac{m-2}{2}y(1-2y)P_{1,m-2}(y).
\]
Hence
\[\mathsf{M}_m E_{m,3}(y)=P_{6,m}(y)-3yP_{3,m-1}(y)-\frac{m-1}{2}yP_{1,m-1}(y)+\frac{m-2}{2}y(1-2y)P_{1,m-2}(y).
\]
Finally, $[y]P_{6,m}(y)=6H_m-m$.
The coefficient of \(y\) in \(-3yP_{3,m-1}(y)\) is \(-3\).  The coefficient of \(y\) in $-\frac{m-1}{2}yP_{1,m-1}(y)$
is \(-\frac{m-1}{2}\), while the coefficient of \(y\) in $\frac{m-2}{2}y(1-2y)P_{1,m-2}(y)$
is \(\frac{m-2}{2}\).  Therefore
\[[y]\mathsf{M}_m E_{m,3}(y)=6H_m-m-3-\frac{m-1}{2}+\frac{m-2}{2}=6H_m-m-\frac72.
\]
This completes the proof.
\end{proof}

We can now determine the magic positivity for \(n=3\).

\begin{prop}\label{prop:n3}
In the range \(m\ge5\), the polynomial \(E_{m,3}(t)\) is magic positive if and only if \(m=5\) or \(m=6\).
\end{prop}
\begin{proof}
First consider \(m=5\) and \(m=6\).  Using the exact formula of Proposition~\ref{prop:M-n3}, we obtain
\[\mathsf{M}_5 E_{5,3}(y)= 1 + \frac{26}{5}y + \frac{451}{30}y^2 + \frac{2779}{120}y^3 +\frac{179}{20}y^4+y^5,
\]
whose coefficients are all positive.  Hence \(E_{5,3}(t)\) is magic positive.
For \(m=6\),
\[\mathsf{M}_6 E_{6,3}(y)=1+\frac{26}{5}y+\frac{1933}{120}y^2+\frac{6601}{240}y^3+\frac{913}{120}y^4+\frac65 y^5+0y^6.
\]
All coefficients are nonnegative.  Hence \(E_{6,3}(t)\) is magic positive.

We next prove non-positivity for \(m\ge7\).  For \(m\ge18\), the first coefficient is already negative.  Indeed, Proposition~\ref{prop:M-n3} gives $[y]\mathsf{M}_m E_{m,3}(y) = 6H_m-m-\frac72$.
For \(m=18\), we have $6H_{18}-18-\frac72=-\frac{360319}{680680}<0$.
Furthermore,
\[\left(6H_{m+1}-(m+1)-\frac72\right)-\left(6H_m-m-\frac72\right)=\frac6{m+1}-1<0
\]
for every \(m\ge6\).  Therefore $6H_m-m-\frac72<0$ with $m\ge18$.
Thus \(E_{m,3}(t)\) is not magic positive for \(m\ge18\).
It remains to check \(7\le m\le17\).  The coefficient of \(y^4\) in the exact product formula of Proposition~\ref{prop:M-n3} is listed in Table \ref{TabCoeff}.
\begin{table}[htbp]
    	\centering
    	\caption{The coefficient of \(y^4\) in $\mathsf{M}_m E_{m,3}(y)$.}
    	\begin{tabular}{c||c|c|c|c}
    		\hline \hline
$m$ & 7 & 8 & 9 & 10   \\
    		\hline
$[y^4]\mathsf{M}_m E_{m,3}(y)$ & $-\frac{79}{210}$ & $-\frac{148643}{10080}$ & $-\frac{2077}{60}$ & $-\frac{5949437}{100800}$  \\
    		\hline
$m$ & 11 & 12 & 13 & 14 \\
            \hline
$[y^4]\mathsf{M}_m E_{m,3}(y)$ &  $-\frac{870026879}{9979200}$ & $-\frac{338188453}{2851200}$ & $-\frac{19864487609}{129729600} $ & $-\frac{693059730551}{3632428800}$ \\
            \hline
$m$ &  15 &  16 & 17 & ---      \\
    		\hline
$[y^4]\mathsf{M}_m E_{m,3}(y)$ &  $-\frac{4212525165859}{18162144000}$ & $-\frac{5030870725201}{18162144000}$ & $-\frac{6300581123911}{19297278000}$ & ---    \\
    		\hline
    	\end{tabular}\label{TabCoeff}
\end{table}
Each listed coefficient is negative.  Hence \(E_{m,3}(t)\) is not magic positive for \(7\le m\le17\).

Combining these cases, \(E_{m,3}(t)\) is magic positive precisely for \(m=5,6\) in the range \(m\ge5\).
\end{proof}

\begin{proof}[Proof of Theorem~\ref{thm:main-Two}]
This result follows directly from Propositions \ref{prop:n1}, \ref{prop:n2-not-positive}, and \ref{prop:n3}.
\end{proof}

\begin{rem}
Theorem~\ref{thm:main-Two} should not be read as a complete classification of magic
positivity for all pairs \((m,n)\) satisfying \(n<m-1\).  It gives a
complete classification for the first three values \(n=1,2,3\), and it
already shows that the behavior below the stable range is mixed:
\(\mathcal{P}(m,1)\) and \(\mathcal{P}(m,2)\) give infinite counterexample families,
whereas \(\mathcal{P}(5,3)\) and \(\mathcal{P}(6,3)\) are positive exceptional cases.
\end{rem}






\noindent
{\small \textbf{Acknowledgments:}}
We are very grateful to Professor Luis Ferroni for discussions on the magic positivity of partial permutohedra.


\begin{thebibliography}{99}
\bibitem{Amanbayeva} A. Amanbayeva and D. Wang, The convex hull of parking functions of length $n$, \emph{Enumer. Comb. Appl.} 2 (2022), no.2, Paper No. S2R10, 10 pp.

\bibitem{Avila26} N. Avila, L. Ferroni, and A. Morales, Luck and magic for Pitman-Stanley polytopes and parking functions, arXiv:2603.19194, (2026).

\bibitem{BeckRobins} M. Beck and S. Robins, \emph{Computing the continuous discretely, in: Integer-Point Enumeration in Polyhedra}, second edition, Undergraduate Texts in Mathematics. Springer, New York, (2015).

\bibitem{Behrend2024} R. Behrend, Ehrhart polynomials of partial permutohedra, arXiv:2403.06975, (2024).

\bibitem{Behedrn2023} R. Behrend, F. Castillo, A. Chavez, A. Diaz-Lopez, L. Escobar, P, Harris, and E. Insko, Partial permutohedra, \emph{S\'em. Lothar. Combin.} 89B (2023), Art. 64, 12pp.

\bibitem{Behedrn2025} R. Behrend, F. Castillo, A. Chavez, A. Diaz-Lopez, L. Escobar, P, Harris, and E. Insko, Partial permutohedra, \emph{Discrete Comput. Geom.} 76 (2026), 589--642.

\bibitem{Black24} A. Black and R. Sanyal, Underlying flag polymatroids, \emph{Adv. Math.} 453 (2024), 109835.


\bibitem{Branden06} P. Br\"and\'en, On linear transformations preserving the P\'olya frequency property, \emph{Trans. Amer. Math. Soc.} 358 (2006), 3697--3716.

\bibitem{Corless96} R. Corless, G. Gonnet, D. Hare, D. Jeffrey, and D. Knuth, On the Lambert W function, \emph{Adv. Comput. Math} 5 (1996), 329--359.

\bibitem{Ehrhart62} E. Ehrhart, Sur les polyh\'edres rationnels homoth\'etiques \'a $n$ dimensions, \emph{C. R. Acad Sci. Paris.} 254 (1962), 616--618.


\bibitem{Ferroni23} L. Ferroni and A. Higashitani, Examples and counterexamples in Ehrhart Theory, \emph{EMS Surv. Math. Sci.} (2024).

\bibitem{Gessel16} I. M. Gessel, Lagrange inversion, \emph{J. Combin Theory, Series A} 144 (2016), 212--249.

\bibitem{Hanada23} M. Hanada, J. Lentfer, and A. Vindas-Mel\'endez, Generalized parking function polytopes, \emph{Ann. Comb.} 28 (2024), 575--613.

\bibitem{Heuer21} D. Heuer and J. Striker, Partial permutation and alternating sign matrix polytopes, \emph{S\'em. Lothar. Combin.} 85B (2021), Art.35, 12pp.

\bibitem{Heuer22} D. Heuer and J. Striker, Partial permutation and alternating sign matrix polytopes, \emph{SIAM J. Discrete Math.} 36 (2022), 2863--2888.

\bibitem{FuLiu19} F. Liu, On positivity of Ehrhart polynomials, Recent trends in algebraic combinatorics, Assoc. Women Math Ser., vol. 16, Springer, Cham, (2019), 189--237.

\bibitem{PostnikovIMRN} A. Postnikov, Permutohedra, associahedra, and beyond, \emph{Int. Math. Res. Not. IMRN} 6 (2009), 1026--1106.

\bibitem{Stanleyh-polynomial} R. P. Stanley, Decompositions of rational convex polytopes, \emph{Ann. Discrete Math.} 6 (1980), 333--342.

\bibitem{RP.Stanley} R. P. Stanley, \emph{Enumerative Combinatorics (volume 1)}, Second Edition. Cambridge Studies in Advanced Mathematics, vol. 49, Cambridge University Press, (2012).

\bibitem{RP.Stanley99} R. P. Stanley, \emph{Enumerative Combinatorics (volume 2)}, second Edition. Cambridge Studies in Advanced Mathematics, vol. 208, Cambridge University Press,  (2023).

\bibitem{Stanley2020AMM} R. P. Stanley, Problem 12191, \emph{Amer. Math. Monthly} 127 (2020), no.6, 563.

\bibitem{Stong2022} R. Stong, The polytope of parking functions, Solution to Problem 12191, \emph{Amer. Math. Monthly} 129 (2022), no.3, 286--289.

\bibitem{Ziegler} G. M. Ziegler, \emph{Lectures on Polytopes}, Grad. Texts Math. 152, Springer-Verlag, New York (1995).

\end{thebibliography}
\end{document}